\documentclass[twoside, english]{amsart} 

\usepackage{amsmath, amsfonts, amsthm, amssymb}
\usepackage{amscd}
\usepackage{color}
\usepackage[utf8]{inputenc}

\renewcommand{\geq}{\geqslant}
\renewcommand{\leq}{\leqslant}

\usepackage{hyperref}

\newtheorem{theorem}{Theorem}

\newtheorem{prop}{Proposition}[section]

\newtheorem*{main-theorem}{Main Theorem}
\newtheorem*{theorem*}{Theorem}

\theoremstyle{definition}

\newtheorem{remark}[prop]{Remark}

\newtheorem*{remark*}{Remark}

\numberwithin{equation}{section}

\def\phi{\varphi}

\def\Re{\,\mathrm{Re}\,}
\def\Im{\,\mathrm{Im}\,}

\def\phi{\varphi}

\def\be{\begin{eqnarray*}}
\def\ee{\end{eqnarray*}}
\def\ben{\begin{eqnarray}}
\def\een{\end{eqnarray}}

\def\L2R{L_{\text{Rest}}^2}

\def\11{\mathds{1}}

\def\L2c{L^2_{\text{comp}}}

\usepackage[T1]{fontenc}
\numberwithin{equation}{section}

\subjclass[2010]{34L20, 35Pxx, 35Q93 58J40, 93Dxx}

\title[Kelvin-Voigt damping]{Decays for Kelvin-Voigt damped wave equations I : the black box perturbative method}
\author{N. Burq}
\address{Laboratoire de Math\'ematiques d'Orsay, Universit\'e Paris-Saclay, CNRS UMR 8628, B\^atiment~307, 91405 Orsay Cedex \& Institut Universitaire de France}

 \email{nicolas.burq@universite-paris-saclay.fr}

 \date{\empty}
\begin{document}

\begin{abstract}   We show in this article how perturbative approaches~\cite{BuHi05} (see also~\cite{AL14}) and the {\em black box} strategy from~\cite{BuZw03} allow to obtain decay rates for  Kelvin-Voigt damped wave equations  from quite standard resolvent estimates : Carleman estimates or geometric control estimates for Helmoltz equationCarleman   or other resolvent estimates for the Helmoltz equation. Though in this context of Kelvin Voigt damping, such approach is unlikely to allow for the optimal results when additional geometric assumptions are considered (see~\cite{BuCh, Bu19}), it turns out that  using this method, we can obtain the usual logarithmic decay which is optimal in general cases. We also present some applications of this approach giving decay rates in some particular geometries (tori).

\ \vskip .1cm 
\noindent {\sc R\'esum\'e.}  On montre dans cet article comment l'approche perturbative inspirée par~\cite{BuHi05} (voir aussi~\cite{AL14}) peut être combinée avec  une stratégie de type  {\em boite noire} de\cite{BuZw03} pour obtenir des taux de décroissance pour les équations d'ondes avec amortissement de Kelvin-Voigt, \`a partir d'estimations de résolvante tr\`es standard: estimées de Carleman ou estimées de contr\^ole géométrique pour l'équation de Helmoltz. Bien que dans ce contexte d'amortissement de Kelvin-Voigt, on ne s'attende pas en général à ce que les résultats obtenus par une telle approche soient optimaux en termes de taux de décroissance (voir~\cite{BuCh, Bu19}), dans le cas le plus général (c'est à dire sans aucune hypothèse géométrique sur la fonction d'amortissement, on obtient bien le résultat optimal. On présente aussi des applications de cette approche donnant des résultats dans le cas de géométries particulières (tores).

\end{abstract} 

\ \vskip -1cm \hrule \vskip 1cm 
 \maketitle

{ \textwidth=4cm \hrule}

\section{Introduction}
\label{S:intro}
In  this paper we are interested in decay rates for Kelvin-Voigt damped wave equations. 
  
  We work in a smooth bounded domain $\Omega \subset \mathbb{R}^d$ and consider the following equation 
\begin{equation}
\left\{ \begin{aligned} &(\partial_t ^2 - \Delta) u - \text{ div} a(x) \nabla_x  \partial_t u =0 \\
& u \mid_{t=0} = u_0 \in H^1_0( \Omega),  \quad \partial_t u \mid_{t=0} = u_1 \in L^2 ( \Omega)\\
&u \mid_{\partial \Omega}  =0 \end{aligned}
\right. 
\end{equation}
with a non negative damping term $a(x)$. The solution can be written as 
\begin{equation}
 U(t) = \begin{pmatrix} u \\ \partial_t u \end{pmatrix} = e^{\mathcal{A} t } \begin{pmatrix} u_0 \\ u_1 \end{pmatrix},
 \end{equation}
 where the generator $\mathcal{A}$ of the semi-group is given by 
 $$ \mathcal{A} = \begin{pmatrix} 0 & 1 \\ \Delta &  \text{ div} a \nabla \end{pmatrix}   $$
 with domain 
 $$ {D}( \mathcal{A}) = \{  (u_0, u_1) \in H^1_0\times L^2;  \Delta u_0 +\text{ div} a \nabla u_1 \in L^2; u_1 \in H^1_0\}.$$ 
The energy of solutions 
$$E((u_0, u_1) (t) = \int_{\Omega} (|\nabla_x u | ^2 + |\partial_t u| ^2) dx $$
satisfies 
$$ E(u_0, u_1) (t) - E(u_0, u_1) (0) = - \int_{0}^t\int_{\Omega} a(x) |\nabla_x \partial_t u |^2 (s,x) ds$$
It was proved in~\cite[Theorem 3]{BuCh} (see also~\cite{LiuRa, Te16} for related results) that if $a$ is smooth, vanishing nicely and the region $\{ x\in \Omega; a(x) >0\}$ controls geometrically $\Omega$, then the rate of decay of the energy is exponential 
$$ \exists c, C>0; \forall (u_0, u_1) \in H^1_0 \times L^2, E((u_0, u_1)) (t) \leq C e^{-ct} E(u) (0).
$$
In this article, we investigate different cases where we can obtain a decay rate  using the perturbative methods from~\cite{BuHi05, BuCh}. In this setting we prove the usual (optimal if the trapping is strong, see Appendix~\ref{app.B}) logarithmic decay assuming only that the damping function $a$ is bounded away from $0$ on a non trivial open set (see~\cite{AmHaRo18} for a similar results with piecewise constant dampings and~\cite{RoZh18} for a weaker decay rate under the same assumptions).  We also show that the same perturbation method applies in different situations (geometric control, arbitrary open set in tori) and give polynomial rates of decay depending on the geometric assumptions considered (see also~\cite{Te16, Zh18}). Our main results are the following

\begin{theorem}\label{th.2}Let $\Omega\subset \mathbb{R}^d$ be a smooth ($C^2$) domain
Assume that the damping term $a\geq 0$ is in  $ L^\infty( {\Omega})$  is bounded away from $0$ on an open set ~$\omega \subset\Omega$ i.e. satisfies
\begin{equation}\label{carre}
\exists \delta>0; \exists \omega\subset \Omega \text{ open };  \forall x  \in \omega,  a(x) \geq \delta.
\end{equation}

Then for any $k$ there exists $C>0$ such that for any $(u_0, u_1) \in D( \mathcal{A}^k)$
$$ E(u_0, u_1) (t) \leq \frac C {\log(t) ^{2k}} \| (u_0, u_1)\|_{D( \mathcal{A}^k)}^2
$$ 

\end{theorem}
\begin{theorem}\label{th.0}
Let $\Omega$ be a smooth ($C^3$) domain. Assume that the $0\leq a \in L^\infty({\Omega})$ satisfies 
\begin{itemize}
\item[A1] There exists $\delta>0$ such that the interior of the set $\omega= \{a>\delta\}$, controls geometrically $\Omega$. i.e. all rays of geometric optics (straight lines) reflecting on the boundary according to the laws of geometric optics eventually reach the set $\omega$ in finite time.
\end{itemize}
Then the energy decays polynomially
 $$\exists C, c>0; \forall (u_0, u_1) \in D(\mathcal{A}^k ),   E(u_0, u_1)) (t) \leq \frac C {(1+ |t|)^k}\|(u_0, u_1) \|^2_{D(\mathcal{A}^k)}.$$
\end{theorem}
\begin{remark}
In the statement above, we did not include the assumption that $\partial \Omega$ has no maximal order tangency contact with its tangents as it is {\em unnecessary} for Melrose and Sj\"ostrand propagation of singularities resuls~\cite{MeSj78}, and hence, for the geometric control condition to imply exact controlability. It is usually added to ensure {\em uniqueness} of generalised bicharacteristics (and hence the existence of a flow)
\end{remark}
\begin{remark}
It is worthwhile to compare this result with the analogous result with classical damping $a(s) \partial_t u$. In this latter case, with a seemingly {\em weaker} damping, under geometric control condition, a {\em stronger} exponential decay is achieved~\cite{BaLeRa92}. This phenomenon is due to {\em overdamping} as explained in~\cite{LiZh16,LiLiZh17}. These works highlight the fact that the smoothness of the damping function plays an important role in achieving strong stabilisation (see also~\cite{BuHi05} where the same phenomenon arises in a different context).  
\end{remark}
In the particular case of the rectangles (in dimension $2$) or hypercubes (in higher dimensions), and Dirichlet boundary conditions, the geometric control condition can be dropped (leading to a slower decay). 
 \begin{equation}
\left\{ \begin{aligned} &(\partial_t ^2 - \Delta) u - \text{ div} a(x) \nabla_x  \partial_t u =0 \\
& u \mid_{t=0} = u_0 \in H^1( \Omega),  \quad \partial_t u \mid_{t=0} = u_1 \in L^2 ( \Omega)\\
&u \mid_{\partial \Omega}  =0 \end{aligned}
\right. 
\end{equation}
\begin{theorem}\label{th.1}Let $\Omega= \prod_{i=1}^n(0, a_i) \subset \mathbb{R}^d$ be a cube in $\mathbb{R}^d$. 
Assume 
\begin{itemize}
\item If $n>2$, the non negative damping term $a\in L^\infty$ is bounded away from $0$ on an open set ~$\omega \subset\Omega$ (i.e. satisfies~\eqref{carre} 
\item If $n=2$, the non negative damping term $a\in L^\infty$ is non trivial  (i.e. it is enough to assume $ \int_\Omega a(x) dx >0$ which is weaker than~\eqref{carre}).
\end{itemize} 

Then for any $k\in \mathbb{N}$, there exists $C$ such that for any $(u_0, u_1) \in D(\mathcal{A}^k)$ 
$$E((u_0, u_1)) (t) \leq \frac{ C} { (1+ |t|)^{\frac k 2}} \| (u_0, u_1)\|^2_{D( \mathcal{A} ^k)}.$$
\end{theorem}
\begin{remark}
While the decay rate in Theorem~\ref{th.2} is in general optimal (i.e. we can construct examples of geometries where it is saturated, see Appendix~\ref{app.B}), the decay rates in Theorems~\ref{th.0} and~\ref{th.1} are unlikely to be optimal. They can actually be improved  for piecewise smooth damping functions~$a$ (see ~\cite{LiuRa05} in dimension $1$), and the methods developed in these works should also apply to the case of tori considered in Theorem~\ref{th.1} (see~\cite{BuHi05})
\end{remark}
\begin{remark}\label{rem.fredholm}
Throughout this note, we shall prove that some operators of the type $P- \lambda \text{Id}$, $\lambda \in \mathbb{R}$ (resp.  $\lambda \in i \mathbb{R}$) are invertible with estimates on the inverse. All these operators share the feature that they have compact resolvent, i.e. $\exists z_0 \in \mathbb{C}; (P- i)^{-1}$  exists and is compact. 
As a consequence, since 
$$ (P- \lambda) = (P-z_0) \Big( \text{Id} + (P-z_0)^{-1} (z_0- \lambda) \Big), $$
and  $( \text{Id} + (P-z_0)^{-1} (z_0- \lambda)^{-1})$ is Fredholm with index $0$,  to show that $(P- \lambda) $ is invertible with inverse bounded in norm by $A$ , it is enough to bound the solutions of $(P- \lambda) u =f$ and prove 
$$ (P- \lambda) u =f\Rightarrow \| u \| \leq M \| f\| .$$
\end{remark}
\begin{remark} An elementary reflection principle shows that considering periodic boundary conditions on the torus $\prod_{i} (0, 2 a_i)$  contains as a special case the case of the  cube $\prod_{i} (0,  a_i)$ with Dirichlet or  Neumann boundary conditions. We  restricted the analysis to the case of Dirichlet boundary conditions to avoid technicalities due to the $0$ frequency. The general case of periodic boundary conditions (and hence also of Neumann boundary conditions)  could be dealt with following the method in~\cite[Appendix B]{BuGe18}
\end{remark}
\begin{remark} Except for the results on tori (Theorem~\ref{th.1}) which are specific to the choice of the flat metric for the Laplace operator, and the last uniqueness result in the proof of Proposition~\ref{resolv.1} which uses the analyticity of the metric (and is only used in Theorem~\ref{th.0} in space dimension $2$), all the results in this paper are also true for non constant coefficients (with the same proof). Actually, the only ingredient we use is basically that the operators $\Delta$ and $\text{div} a \nabla$ are symetric for the {\em same} integration measure. We could e.g. replace $\Delta $ by $\frac 1 {\kappa (x) } \text{div} g^{i,j} (x) \nabla_x $ and $\text{div} a \nabla$  by $\frac 1 {\kappa (x) }  \text{div} a^{i,j} \nabla$ with $\kappa \geq c>0$, $(g^{i,j})$ positive definite and $(a^{i,j})$ non negative. Dropping the analyticity of the metrics $g$ in this case would require reinforcing the non triviality assumption in Proposition~\ref{resolv.1} to the stronger~\eqref{carre} (which is in any case required for Theorems~\ref{th.2} and~\ref{th.1}).
\end{remark}
It is well known that decay estimates for the evolution semi-group follow from resolvent estimates~\cite{Bu98, RTT05, BoTo10, BaDu08, TuWe09}. The plan of the paper is the following. In Section~\ref{sec.2} we show that as soon as the damping term $a\in L^\infty$ is not equal to $0$ almost everywhere, these resolvent estimates hold true in the low (bounded) frequency regime. Then in all cases the strategy is the same: we work on the resolvent and first prove an {\em a priori} estimates on the support of $a$ which allows to put in the r.h.s. the (non perturbative) term coming from Kelvin-Voigt damping and treat it {\em as a perturbative source term} for the Helmoltz equation. Then we apply {\em ad nauseam} variations around quite standard resolvent estimates for the Helmoltz equation that we use as a {\em black box} and which depend on the geometry of the problem studied. In Section~\ref{sec.3}, to prove the high frequency resolvent estimates for Theorem~\ref{th.2}, the standard Carleman estimates are used. In Section~\ref{sec.4}, to prove Theorem~\ref{th.0}, we apply  the geometric control estimates. Finally in Section~\ref{sec.5},  the torus resolvent estimates from~\cite{AnMa14,   BuZw12} are involved.
For completeness in Appendix~\ref{app.A} we give a proof of the classical geometric control resolvent estimates which is at the heart of the proof of Theorem~\ref{th.0}, and in Appendix~\ref{app.B} we give an example where the decay rate in Theorem~\ref{th.2} is optimal. See also~\cite{Le96}

{\bf Acknowledgments:} The author is supported by Agence nationale de la recherche grant  "ISDEEC'' ANR-16-CE40-0013. I would also like to thank referees whose comments improved the article.
 \section{Low frequencies}	\label{sec.2}
The purpose of this section is to prove that for low frequencies $\lambda$ the resolvent of the operator $\mathcal{A}$ is bounded.
\begin{prop} \label{resolv.1}Assume that $a\in L^\infty$ is non negative  $a\geq 0$ and non trivial $\int_\Omega a(x) dx >0$).
Then for any $M>0$, there exists $C>0$ such that for all $\lambda \in \mathbb{R}, |\lambda | \leq M$, the operator $\mathcal{A} - i \lambda$ is invertible from $D( \mathcal{A})$ to $\mathcal{H}$ with estimate
\begin{equation}\label{resolv}
\| ( \mathcal{A} - i \lambda)^{-1} \|_{\mathcal{L}( \mathcal{H})} \leq C .
\end{equation}
\end{prop}
We start with $\lambda=0$ where
\begin{equation}\label{frequzero} 
\begin{aligned} \mathcal{A} \begin{pmatrix} u\\v\end{pmatrix} = \begin{pmatrix}f\\ g\end{pmatrix} &\Leftrightarrow \left\{ \begin{aligned} & v = f ,\\
&\Delta u + \text{div} a \nabla_x v= g, \end{aligned}\right.\\
&  \Leftrightarrow \left\{ \begin{aligned} & v = f,  \\
&\Delta u=- \text{div} a \nabla_x f + g,  \end{aligned} \right.
\end{aligned}
\end{equation}
and the result follows from the inversion of the Laplace operator with Dirichlet boundary conditions from $H^{-1}$ to $H^1$. 
Let us now study the case $\lambda \neq 0$. 

\begin{equation}\label{above} \begin{aligned}(\mathcal{A} - i \lambda)\begin{pmatrix} u\\v\end{pmatrix} = \begin{pmatrix}f\\ g\end{pmatrix} &\Leftrightarrow \left\{ \begin{aligned} &-i \lambda u + v = f \\
&\Delta u + \text{div} a \nabla_x v- i \lambda v= g \end{aligned}\right.\\
 &\Leftrightarrow \left\{ \begin{aligned} & v = f +i \lambda u \\
&-i \lambda^{-1}\Delta v + \text{div} a \nabla_x v -i \lambda v= g +i \lambda^{-1} \Delta f \end{aligned} \right.
\end{aligned}
\end{equation}
Following~\cite{AmHaRo18}, we remark that the operator $A= i \lambda^{-1}\Delta  + \text{div} a \nabla_x$ is by Lax Milgram Theorem invertible from $ H^1_0$ to $H^{-1}$ hence we can write 
$$(-i \lambda^{-1}\Delta  + \text{div} a \nabla_x  -i \lambda ) = ( \text{ Id} - i \lambda A^{-1}) A$$
and consequently  the operator 
$$P=-i \lambda^{-1}\Delta  + \text{div} a \nabla_x  -i \lambda $$
is invertible from $H^1_0$ to $H^{-1}$ iff the operator  $( \text{ Id} - i \lambda A^{-1}) $ is invertible on $H^{-1}$. But $A^{-1}$ is continuous from $H^{-1}$ to $H^1_0$ hence compact on $H^{1} $ and $( \text{ Id} - i \lambda A^{-1}) $ is Fredholm of index $0$ hence invertible iff it is injective. This shows that $P$ is invertible from $H^1_0$ to $H^{-1}$ iff it is injective.

Coming back to~\eqref{above} we get easily that $(\mathcal{A} - i \lambda)$ is invertible from $H^1_0 \times L^2 $ to $D( \mathcal{A})$ iff it is injective, and consequently to invert the operator $\mathcal{A} - i \lambda$, and estimate the norm of the inverse, it is enough to estimate solutions to the system~\eqref{above}.

We now show that for any $M$, ~\eqref{resolv} holds uniformly for $\lambda\in [-M, M]$. Otherwise, there would exist sequences $(U_n) = (u_{n}, v_n) $ and $F_n = (f_n, g_n) $ in $ H^1_0 \times L^2$, $\lambda_n \in [-M, M]$ such that 
$$ (\mathcal{A} - i \lambda_n ) U_n = F_n, \| U_n \|_{H^1\times L^2} =1, \| F_n \|_{H^1\times L^2} \rightarrow_{n\rightarrow + \infty} 0.$$
Extracting a subsequence, we can assume that $U_n$ converges weakly to $U= (u,v)$ in $H^1\times L^2$ and that $\lambda_n \rightarrow \lambda \in [-M. M]$. 
Passing to the limit in~\eqref{above} (for $U_n, F_n, \lambda_n$) we get 
\begin{equation}\label{above2}
 v= i \lambda u, \qquad \Delta u +i \lambda \text{div} a \nabla_x u +\lambda^2 u=0 \in H^{-1}.
\end{equation}
Multiplying the  last line in~\eqref{above}  (for $U_n, F_n, \lambda_n)$) by $\overline{u_n}$ integrating by parts and taking the imaginary part gives
\begin{multline}\label{apriori3}
|\lambda_n | \int _{\Omega} a(x) |\nabla_x u_n|^2 dx= \pm \Im \bigl( g_n - \text{div} a \nabla_x f_n + i \lambda f_n , u_n\bigr)_{L^2} \\
\leq   \| g_n + \text{div} a \nabla_x f _n+ i \lambda f_n\|_{H^{-1}} \| u_n\|_{H^1_0}   \rightarrow_{n\rightarrow + \infty}0, 
\end{multline}
which implies 
$$|\lambda | \int _{\Omega} a(x) |\nabla_x u|^2 dx= 0.$$
 Multiplying the  last line in~\eqref{above} (for $U_n, F_n, \lambda_n)$), by $\overline{u_n}$ integrating by parts and taking the real part gives
\begin{equation}\label{apriori4}
 \int _{\Omega}  \lambda_n^2 |u_n|^2- |\nabla_x u_n|^2 dx= - \Re \bigl( g_n - \text{div} a \nabla_x f_n + i \lambda f_n , u_n\bigr)_{L^2}   \rightarrow_{n\rightarrow + \infty}0
\end{equation}
Since $u_n$ converges to $u$ strongly in $L^2$, we deduce that $v_n = i\lambda_n u_n$ converges also strongly to $v= i\lambda u$ in $L^2$ and from~\eqref{above2}and ~\eqref{apriori4}
$$ \| \nabla_xu\|_{L^2}^2 = \lambda^2 \| u\|_{L^2}^2= \lim_{n\rightarrow + \infty} \lambda_n^2 \| u_n\|_{L^2}^2=\lim_{n\rightarrow + \infty} \|\nabla_x u_n\|_{L^2}^2,$$
and consequently $u_n$ converges strongly to $u$ in $H^1$. 

Finally, if $\lambda \neq 0$, using ~\eqref{apriori3} we get that $\Delta u + \lambda^2 u =0$,   in $\Omega$ and $a \nabla_x u$ vanishes. As a consequence, if ~\eqref{carre} is assumed, the function $u$ is constant on $\omega$ hence (using the equation $(\Delta + \lambda^2 ) u =0$), $u$ vanishes on $\omega$ and using the uniqueness result for solutions to second order elliptic PDE's $u$ is identically $0$. If the weaker $\int_\Omega  a dx >0$ is assumed, we use in that case that $u$ is analytic in $\Omega$. Then since $\nabla_x u$ vanishes on the support of $a$ which has non zero measure, we deduce from analyticity that $\nabla_xu$ is identically $0$ and $u$ is constant everywhere,  and $v= \lambda u= - \Delta u$ is also equal to $0$.   On the other hand, if $\lambda =0$, then $v= \lambda u =0$, $\Delta u=0, u \in H^1_0$ hence $u=0$. In both cases this contradicts $\| U\|_{H^1\times L^2} = \lim_{n\rightarrow + \infty } \| U_n \|_{H^1\times L^2} =1$.

\section{High energy estimates for general dampings}\label{sec.3}
Theorem~\ref{th.1} follows (see~\cite[Théorème 3]{Bu98}) from  the low frequency resolvent estimate Proposition~\ref{resolv.1} and  the following high energy resolvent estimate

\begin{prop}\label{resolv.2} 
Assume that $a\in L^\infty$ satisfies $a\geq c>0$ on an open set $\omega$.
Then  there exists $M, C, c>0$ such that for all $\lambda \in \mathbb{R}, |\lambda | \geq M$, the operator $\mathcal{A} - i \lambda$ is invertible from $D( \mathcal{A})$ to $\mathcal{H}$ with estimate
\begin{equation}\label{resolvbis}
\| ( \mathcal{A} - i \lambda)^{-1} \|_{\mathcal{L}( \mathcal{H})} \leq C e^{c |\lambda|} .
\end{equation}
\end{prop}
We shall deduce this estimate from the following refinement of Carleman estimates for solutions to Helmoltz equations 
\begin{prop} \label{carle}
Assume that $b \in C^\infty_0 ( \Omega)$ is not identically $0$. Then there exists $C,c>0$ such that for any $(u,f) \in H^1_0\times H^{-1} ( \Omega)$, for any $\tau \in \mathbb{R}$,
$$ (\Delta + \tau^2)u =f \Rightarrow  \| u\| _{H^1} \leq C e^{c|\tau|} \bigl( \| f\|_{H^{-1} } + \| bu\|_{L^2} \bigr).
$$
\end{prop}
\begin{proof}
Let $\widetilde{b} \in C^\infty_0 ( \Omega)$ non trivial and such that $|b|\geq c>0$ on the support of $\widetilde{b}$. We first notice that classical Carleman observation estimate (see~\cite{LeRo95, Bu98}) imply 
 \begin{equation}\label{estim}
  (\Delta + \tau^2)u =f \Rightarrow  \| u\| _{H^1} \leq C e^{c|\tau|} \bigl( \| f\|_{L^2 } + \| \widetilde{b}u\|_{H^1} \bigr).
\end{equation}
In~\eqref{estim} we can now  replace the r.h.s. by 
\begin{equation}\label{carleman}
C e^{c|\tau|} \bigl( \| f\|_{L^2 } + \| {b}u\|_{L^2} \bigr).
\end{equation}
Indeed, we have 
\begin{multline}
( \Delta u + \tau^2 u, \widetilde{b}^2 u)_{L^2} = (f,\widetilde{b}^2 u)_{L^2} \\
\begin{aligned}
 & \Rightarrow \| \widetilde{b}\nabla_x u\|_{L^2}^2 = \tau^2 \| \widetilde{b}u\|_{L^2}^2 - (2\widetilde{b} \nabla_x \widetilde{b} \cdot \nabla_x u , u)_{L^2} - ( \widetilde{b}f, bu)\\
& \Rightarrow \| \widetilde{b}\nabla_x u\|_{L^2}^2 \leq C \bigl( \tau^2 \| \widetilde{b}u\|_{L^2}^2+ \| {b} u\|_{L^2}^2 + \| \widetilde{b}f \|_{H^{-1}}^2 \bigr) \leq C (1+ |\tau|^2) \| bu \|_{L^2}^2,
\end{aligned}
\end{multline}
since $|b| \geq c>0$ on the support of  $\nabla_x \widetilde{b} $. 
Now, we consider $(v,g)$ solutions to 
$$ (\Delta + \tau^2 \pm i{b}^2)v =g.  $$
We have 
$$ \pm \Im \bigl( (\Delta + \tau^2 \pm i{b}^2)v, v\bigr)_{L^2} = \| {b} v\|_{L^2} ^2 = \bigl(\pm g, v\bigr) \leq \| g\|_{H^{-1} } \| v\|_{H^1}, $$
and from~\eqref{carleman} we deduce 
\begin{equation}
\begin{aligned} \| v\|_{H^1} & \leq Ce^{c|\tau|} \bigl( \| \mp i{b}^2 v + g \|_{L^2} + \| {b}v\|_{L^2} \bigr),\\ &\leq Ce^{c\|\tau|} \bigl( \|  g \|_{L^2} +  \| g\|^{1/2}_{H^{-1} } \| v\|^{1/2}_{H^1} \bigr),\\
\Rightarrow  \| v\|_{H^1} \hfill & \leq Ce^{c|\tau|} \| g\|_{L^2} + \frac 1 2 \bigl( C^2 e^{2c|\tau}   \| g\|_{H^{-1} } + \| v\|_{H^1} \bigr),\\
\Rightarrow  \| v\|_{H^1} \hfill & \leq C'e^{2c|\tau|} \| g\|_{L^2}, 
\end{aligned}
\end{equation}which shows that  the inverse of the operator $(\Delta + \tau^2 \pm i{b}^2)$ is bounded from $L^2$ to $H^1_0$ by $C' e^{2c|\tau|}$.  By duality $(\Delta + \tau^2 \pm i {b}^2)^{-1}$ is bounded from  $H^{-1}$ to $L^2$ and using the equation from $H^{-1} $ to $H^1_0$ (by $C (1+|\tau|^2) e^{2c|\tau|}$). Coming back to the proof of Proposition~\ref{carle}
, since
$$ (\Delta + \tau^2 + i\widetilde{b}^2) u = f+i\widetilde{b}^2 u,$$
we get 
$$\| u\| _{H^1} \leq C e^{c'|\tau|} \bigl( \| f+ i \widetilde{b} u\|_{H^{-1} } + \| \widetilde{b}u\|_{H^{-1}}\bigr) \leq C e^{c'|\tau|} \bigl( \| f\|_{H^{-1} } + \| \widetilde{b}u\|_{L^2}\bigr)
$$
\end{proof}
We now can prove Theorem~\ref{th.2} 
and come back to the case $|\lambda| >M$. 
Multiplying the  last line in~\eqref{above} by $\overline{u}$ integrating by parts and taking the imaginary part gives
\begin{multline}\label{apriori2}
|\lambda| \int _{\Omega} a(x) |\nabla_x u|^2 dx= \pm \Im \bigl( g - \text{div} a \nabla_x f + i \lambda f , u\bigr)_{L^2} \\
\leq   \| g + \text{div} a \nabla_x f + i \lambda f\|_{H^{-1}} \| u\|_{H^1_0}
\leq C(1+ |\lambda|) \bigl( \| g\|_{L^2} + \| f\|_{H^1} \bigr) \| u \|_{H^1}.
\end{multline}
On the other hand we can choose a  non trivial smooth function $b\in C^\infty_0 ( \Omega)$ such that  $0\leq b\leq a$ (recall that $a\geq \delta$ on an open set $\omega$), and multiplying the  last line above in~\eqref{above} by $b^2 \overline{u}$ integrating by parts and taking the real  part gives
\begin{equation}
\lambda^2 \|bu\|^2_{L^2} = \| b\nabla_x u\|^2_{L^2} + \Im \bigl(a \nabla_x u, \nabla_x ( b^2 u) \bigr)_{L^2} + \Re \bigl(g+ i \lambda f , b^2 u\bigr),
+\Re \bigl( a \nabla_x f, \nabla_x (b^2 u)\bigr),
\end{equation}
from which we deduce (using $0\leq b \leq a$)
\begin{equation}
\begin{aligned}\lambda^2 \|bu\|^2_{L^2} &\leq C  \Bigl(\| a\nabla_x u\|^2_{L^2} + C\|a \nabla_x u\|_{L^2} \| bu\|_{L^2}\Bigr)\\
\hfill & +\bigl(\|g\|_{L^2}+  |\lambda|\| f\|_{L^2}) \|b u\|_{L^2} +\|f\|_{H^1}(\| a \nabla_x u\| + \|bu\|_{L^2}) \bigr),\\
\Rightarrow  \|bu\|^2_{L^2} &\leq C'  \Bigl(\| a\nabla_x u\|^2_{L^2} + (1+ |\lambda|) (\|g\|^2_{L^2}+  \| f\|^2_{H^1}) \Bigr).
\end{aligned}
\end{equation}

Since
$$\Delta u  +\lambda^2 u= g + \text{div} a \nabla_x f + i \lambda f -i \lambda \text{div} a \nabla_x u,
$$from Proposition~\ref{carle} we get 
\begin{multline}
\| u\|_{H^1} \leq C e^{c |\lambda|} \bigl(\| g + \text{div} a \nabla_x f + i \lambda f -i \lambda \text{div} a \nabla_x u\| _{H^{-1}}\bigr) + \| bu\|_{L^2}\\
\leq C (1+ |\lambda|) e^{c |\lambda|} \bigl( \| f\|_{H^1} + \| g\|_{L^2} + \| a \nabla_x u\|_{L^2} \bigr), \\
\leq C (1+ |\lambda|) e^{c |\lambda|} \bigl( \| f\|_{H^1} + \| g\|_{L^2} + C(1+ |\lambda|)^{1/2} \bigl( \| g\|_{L^2} + \| f\|_{H^1} \bigr)^{1/2} \| u \|^{1/2}_{H^1} \bigr),
\end{multline}
which implies 
$$ \| u\|_{H^1} \leq C'(1+ |\lambda|)^{3} e^{2c |\lambda|} \bigl(\| f\|_{H^1} + \| g\|_{L^2}\bigr),
$$ 
which ends the proof of Proposition~\ref{resolv.2} and hence of Theorem~\ref{th.2}.
\section{General results under geometric control assumption}\label{sec.4}
According to~\cite{BoTo10, BaDu08}  and Proposition~\ref{resolv.1}, to prove Theorem~\ref{th.0} it is enough  to prove the following high energy resolvent estimate
\begin{prop}\label{resolv.3} Assume that $a\in L^\infty$ satisfies $a\geq c>0$ on an open set $\omega$, and the geometric control assumption of Theorem~\ref{th.0} is satisfied.
Then  there exists $M, C, c>0$ such that for all $\lambda \in \mathbb{R}, |\lambda | \geq M$, the operator $\mathcal{A} - i \lambda$ is invertible from $D( \mathcal{A})$ to $\mathcal{H}$ with estimate
\begin{equation}\label{resolvter}
\| ( \mathcal{A} - i \lambda)^{-1} \|_{\mathcal{L}( \mathcal{H})} \leq C \lambda^2
\end{equation}
\end{prop}

\begin{proof}{Proof of Proposition~\ref{resolv.3}}
 We shall use as a black box the following resolvent estimates for the Helmoltz operator

\begin{prop}\label{helmoltz-geo}Under the geometric control assumptions of Theorem~\ref{th.0},  there exists $C>0$ such that for all $\lambda \in \mathbb{R}, |\lambda | \geq 1$, for any solutions $(u,f,g) \in H^1_0 \times L^2 \times H^{-1}$ of 
$$( \Delta+ \lambda^2) u = f+ g,
$$
we have
$$ \| u\|_{L^2} + \frac 1 {1+ |\lambda|} \| u\|_{H^1} \leq \frac C{ 1+ |\lambda|}  \| f\|_{L^2}  + C \| g\|_{H^{-1}} + C \| u 1_{\omega} \|_{L^2}.
$$
\end{prop}
This result (with only the $L^2$ norm in the l.h.s. and with $g=0$) is reminiscent in the folklore of the propagation of semi-classical singularities litterature (see~\cite[Sections 3\&4]{Bu02} for a slightly different version). It can be shown that it is actually essentially equivalent to the geometric control condition on the domain $\omega$.   To be complete, we explain in Appendix~\ref{app.A} how to deduce it from the Bardos-Lebeau-Rauch exact controlability result~\cite{BaLeRa92}. 

We take as $H^1_0$ norm $\| u\|_{H^{1}_0} = \| \nabla _x u \|_{L^2}$ and we argue by contradiction. If~\eqref{resolvter} were not true, there would exist sequences $\lambda_n \rightarrow \pm \infty$, $(u_n, v_n)\in D( \mathcal{A}), (f_n, g_n ) \in H^1_0 \times L^2$ such that 
\begin{equation}
\label{contradict}
( \mathcal{A} - i \lambda_n ) \begin{pmatrix} u \\v \end{pmatrix} = \begin{pmatrix} f\\ g \end{pmatrix}, \quad  \|(u_n, v_n)\|_{ D( \mathcal{A})} =1, \quad \|  (f_n, g_n )\|_{ H^1_0 \times L^2} = o( \lambda_n ^{-2}).
\end{equation} 
From now, we assume for simplicity  that $\lambda_n \rightarrow + \infty$ denote by $h_n = \lambda_n^{-1} $ and drop for conciseness the index $n$. We have 
\begin{equation}\label{a-priori}
 \Re \Bigl( (\mathcal{A} - i \lambda ) \begin{pmatrix} u \\v \end{pmatrix}, \begin{pmatrix} u \\v \end{pmatrix}\Bigr)_{H^1_0\times L^2} 
 = \int _\Omega a (x) | \nabla_x v |^2 dx
\leq \| (u, v) \|_{H^1_0 \times L^2} \| (f,g) \|_{H^1_0 \times L^2} = o(h^2). 
\end{equation}
From~\eqref{above}, we have 
$$ i \lambda v= \Delta u + \text{div} a \nabla_x v- g,  $$
which from~\eqref{a-priori} implies that on any subdomain $\omega \subset \Omega$ 
$$ \| v \|_{H^{-1} ( \omega)} = O(h) + o(h^2).$$
If $a\geq c>0$ on $\omega$ as in Proposition~\ref{resolv.3}, by interpolation with~\eqref{a-priori},  we get 
$$ \| v \|_{L^2( \omega)} = o(h).$$

Again from ~\eqref{above}, we have 
\begin{equation}\label{droite}  \left\{ \begin{aligned} & v = f +i \lambda u \\
&\Delta v +  \lambda^2 v= g +i \lambda^{-1} \Delta f - i \lambda \text{div} a \nabla_x v\end{aligned} \right.
\end{equation}
which implies from Proposition~\ref{helmoltz-geo} 
\begin{multline}
\| v\|_{L^2} + \frac 1 {1+ |\lambda|} \| v\|_{H^1} \leq \frac C{ 1+ |\lambda|}  \| g\|_{L^2}  + C \| ih \Delta f -i h^{-1} \text{ div } a \nabla v  \|_{H^{-1}} + C \| v \|_{L^2(\omega)}\\
\leq \frac C{ 1+ |\lambda|}  \| g\|_{L^2}  + C \|h f\|_{H^1} +C  h^{-1} \| a \nabla v  \|_{L^2} + C \| v \|_{L^2(\omega)}\\ \leq o(h^3) +o(h^3) + o(1) + o(h) 
\end{multline}
This, combined with the relation $v= f+ i \lambda u$ (from~\eqref{above} implies that 
$$\| (u,v) \|_{H^1\times L^2} = \| (ih (f-v),v) \|_{H^1\times L^2} = o(1) $$ leading to a contradiction with the normalization of $(u,v)$ in $H^1\times L^2$ norm (see~\eqref{contradict})
\end{proof}
\begin{remark}
In the proof above, the non optimality appears clearly in the fact that the only term for which we do not have any slack is the term $\lambda \text{div} a \nabla_x v$ that we put in the r.h.s of the Helmoltz equation. To improve the result, we need to keep this term in the left. However, in this case this term becomes delicate to handle and this strategy requires some smoothness on $a$ (see~\cite{LiuRa, BuCh, Bu19}). 
\end{remark}

\section{Results on tori}\label{sec.5}
In this section we  show how the same approach can give (non optimal) decay rates on tori and prove Theorem~\ref{th.1}.  According to~\cite{BoTo10, BaDu08} (and the low frequency resolvent estimates in Proposition~\ref{resolv.1}), it is enough to prove the following high frequency resolvent estimate
\begin{prop}\label{resolv.4} Assume that $a\in L^\infty$ satisfies $a\geq c>0$ on an open set $\omega$.
Then  there exists $M, C, c>0$ such that for all $\lambda \in \mathbb{R}, |\lambda | \geq M$, the operator $\mathcal{A} - i \lambda$ is invertible from $D( \mathcal{A})$ to $\mathcal{H}$ with estimate
\begin{equation}\label{resolvtquar}
\| ( \mathcal{A} - i \lambda)^{-1} \|_{\mathcal{L}( \mathcal{H})} \leq C \lambda^4
\end{equation}
\end{prop}
The proof of Proposition~\ref{resolv.4} follows line by line the proof of Proposition~\ref{resolv.3} after replacing Proposition~\ref{helmoltz-geo} by the following 
\begin{prop}\label{helmoltz-tore}Under the geometric control assumptions of Theorem~\ref{th.0},  there exists $C>0$ such that for all $\lambda \in \mathbb{R}, |\lambda | \geq 1$, for any solutions $(u,f,g) \in H^1_0 \times L^2 \times H^{-1}$ of 
$$( \Delta+ \lambda^2) u = f+ g,
$$
we have
$$ \| u\|_{L^2} + \frac 1 {1+ |\lambda|} \| u\|_{H^1} \leq C \| f\|_{L^2}  + C (1+ |\lambda|) \| g\|_{H^{-1}} + C \| u 1_{\omega} \|_{L^2}.
$$
\end{prop}
In turn in dimension $2$, Proposition~\ref{helmoltz-tore} follows from. 
\begin{theorem}[\protect{~\cite[Theorem 2]{BuZw18}}]
\label{th.3}
 For $a\in L^2( \mathbb{T}^2)$, $ a \geq 0 $, $ \| a \|_{L^2} > 0$,  there exist $C, c>0$ such that for any $u_0 \in L^2( ,\mathbb{T}^2)$, the equation 
 \begin{equation}\label{damped} (i \partial_t + \Delta + ia) u=0, \qquad u |_{t=0} = u_0, 
\end{equation}
has a unique global solution $ u \in L^\infty( \mathbb{R}; L^2( ,\mathbb{T}^2))\cap L^4( ,\mathbb{T}^2; L^2_{\rm{loc}}(\mathbb{R})) $ and 
\begin{equation}\label{dampedbis}
 \| u\|_{L^2( ,\mathbb{T}^2)} (t) \leq C e^{-ct} \| u_0 \|_{L^2( ,\mathbb{T}^2)}.
 \end{equation}
\end{theorem}
Indeed it is well known that exponential stabilization for the group $e^{it( \Delta +i a) }$ is equivalent to a resolvent estimate
\begin{equation}\label{resolv-tore}
\exists C; \forall z \in \mathbb{R}, \| (\Delta + ia - z)^{-1}\|_{\mathcal{L} (L^2)}\leq C 
\end{equation}

which in turn using the equation $(\Delta + ia - z)u =f$ 
implies 
$$ \| \nabla_x u \|^2_{L^2} \leq C |z| \| u\|_{L^2}^2 ,
$$ hence by duality the resolvent is also bounded from $H^{-1} $ to $L^2$ by $C\sqrt{|z|}$ and using again the equation we get the boundedness from $H^{-1}$ to $H^1$ by $C|z|$. Now  
$$(\Delta   + \lambda^2 ) u = f +g \Leftrightarrow (\Delta +i1_\omega  + \lambda^2 ) u = (f + iau) +g .$$ We deduce 
$$ \| u\|_{L^2} + \frac 1 {1+ |\lambda|} \| u\|_{H^1} \leq C \| f\|_{L^2}  + C(1+ |\lambda|) \| g\|_{H^{-1}} + C \| 1_\omega u  \|_{L^2}.
$$
Finally a simple reflection principle for functions satisfying Dirichlet boundary conditions on the cube $K$ gives periodic functions on the torus $2K$, and Proposition~\ref{helmoltz-tore} follows. 

In higher dimensions, we need to assume that $a\geq \delta >0$ on a small open set $\omega$ for which we can apply the exact controlability result in~\cite{Ko97} (see also~\cite{AnMa14} for general tori),  which in turn imply the stabilisation result~\eqref{dampedbis} (see~\cite[Section 4]{BuZw18}) which as previously implies the resolvent estimate~\eqref{resolv-tore} and the rest of the proof is the same.
\appendix
\section{The propagation estimate Proposition~\ref{helmoltz-geo}} \label{app.A}
We first show that Proposition~\ref{helmoltz-geo} follows from the weaker 
\begin{prop}\label{helmoltz-weak}
Under the geometric control assumptions of Theorem~\ref{th.0},  there exists $C, M>0$ such that for all $\lambda \in \mathbb{R}, |\lambda | \geq M$, for any solutions $(u,f) \in H^1_0 \times L^2 $ of 
$$( \Delta+ \lambda^2) u = f,
$$
we have
$$ \| u\|_{L^2} \leq \frac C{ 1+ |\lambda|}  \| f\|_{L^2}   + C \| u 1_{\omega} \|_{L^2}.
$$
\end{prop}
\begin{proof}
Indeed, from Proposition~\ref{helmoltz-weak} we deduce that the operator 
$$P_{\pm}= (\Delta \pm i1_\omega \lambda + \lambda^2): H^1_0 \rightarrow H^{-1}$$
is invertible with inverse bounded on $L^2$ by $C/(1+ |\lambda|)$. Indeed,
$$|\lambda | \| 1_\omega u \|_{L^2}^2 = |\Im (P_{\pm}u, u)_{L^2} \leq \| Pu \|_{L^2} \| u \|_{L^2},$$
and from Proposition~\ref{helmoltz-weak}, if $P_{\pm}u= f$ (hence $(\Delta + \lambda^2) u = f  \mp ia \lambda u $
$$\| u\|_{L^2} \leq \frac {C} {(1+ |\lambda|} \|f\|_{L^2} + (C+1) \| 1_\omega u\|_{L^2} \leq  \frac {C} {(1+ |\lambda|)} \|f\|_{L^2} + (C+1)\Bigl( \frac{ \| f\|} { \lambda} \Bigr)^{1/2} \| u\|_{L^2} ^{1/2}$$
$$ \Rightarrow \| u\|_{L^2} \leq \frac {C'} {(1+ |\lambda|)} \|f\|_{L^2} .$$
Now 
$$ \bigl|\| \nabla_x u \|_{L^2}^2 - \lambda^2 \|u \|_{L^2}\bigr| = | \Re ( Pu, u)_{L^2}|  \leq \| Pu\|_{L^2} \|u\|_{L^2} $$
and we deduce that $P_{\pm}^{-1} $ is bounded in norm from $L^2$ to $H^1$ by $C'$, hence by duality from $H^{-1} $ to $L^2$ by $C'$ and again using the equation we get the boundedness from $H^{-1} $ to $H^1_0$ by $C' ( 1+ |\lambda|)$. 

Finally, as previously to conclude the proof of Proposition~\ref{helmoltz-geo}, we just write 
$$(\Delta + \lambda^2 ) u = f+ g \Leftrightarrow (\Delta + \lambda^2 + i1_{\omega} ) u= (f+ i1_\omega u) + g,$$
and we conclude using the bounds from $L^2$ to $L^2$ and $H^1$ for the contribution of $(f+ i1_\omega u)$ and from $H^{-1}$ to $L^2$ and $H^1$for the contribution of $g$.
\end{proof}
It remains to show that the geometric control property implies Proposition~\ref{helmoltz-weak}. It follows from the fact that exact controlability is essentially equivalent  (see~\cite[(6.8)]{BuZw03} for the boundary control version and~\cite[Corollary 3.10]{Mi12} for an abstract version while on the other hand from  Bardos-Lebeau-Rauch geometric control result~\cite{BaLeRa92},  it is equivalent to geometric control for wave equations.

We consider the damped wave equation 
\begin{equation}\label{damped-waves}
\left\{ \begin{aligned} &(\partial_t ^2 - \Delta) u - \text{ div} a(x) \nabla_x  \partial_t u =0 \\
& u \mid_{t=0} = u_0 \in H^1( \Omega),  \quad \partial_t u \mid_{t=0} = u_1 \in L^2 ( \Omega)\\
&u \mid_{\partial \Omega}  =0 \end{aligned}
\right. 
\end{equation}
with a non negative damping term $a(x)$, and energy
$$E(u) (t) = \int_{\Omega} (|\nabla_x u | ^2 + |\partial_t u| ^2) dx $$
The solution can be written as 
\begin{equation}
 V(t) = \begin{pmatrix} u \\ \partial_t u \end{pmatrix} = e^{\mathcal{B} t } \begin{pmatrix} u_0 \\ u_1 \end{pmatrix},
 \end{equation}
 where the generator $\mathcal{B}$ of the semi-group is given by 
 $$ \mathcal{B} = \begin{pmatrix} 0 & 1 \\ \Delta &  a  \end{pmatrix}  \begin{pmatrix} u_0 \\ u_1 \end{pmatrix} ,$$
 with domain 
 $$ {D}( \mathcal{B}) =  H^1_0\cap H^2 \times H^1_0. $$ 

\begin{theorem} [Geometric control implies exponential decay for damped wave equations \cite{BaLeRa92}] Assume that the domain $\omega$ controls geometrically $\Omega$. Then the energy of solutions to~\eqref{damped-waves} decays exponentially
$$ E((u_0, u_1)) (t) - E((u_0, u_1)) (0) = - \int_{0}^t\int_{\Omega} a(x) |\nabla_x \partial_t u |^2 (s,x) ds.$$
\end{theorem}
From this we deduce that the resolvent $(\mathcal{B}- i \lambda)^{-1}$ exists for any $\lambda \in \mathbb{R}$ and satisfies 
$$\exists C>0 ; \forall \lambda \in \mathbb{R}. \| (\mathcal{B}- i \lambda)^{-1}\|_{\mathcal{L} ( H^1_0 \times L^2)} \leq C.$$
Applying this estimate to $f=0, g\in L^2$ we get for $(u,v) \in H^2 \cap H^1_0 \times L^2$
$$ -i \lambda u +v =0, \quad \Delta u +a v =g \Rightarrow ( \| u\|_{H^1} +\| v\|_{L^2} ) \leq C \| g\|_{L^2} $$
or equivalently for $u \in H^2 \cap H^1_0$
$$ ( \Delta  + \lambda^2 + i \lambda a)  u =g \Rightarrow ( \| u\|_{H^1} + \lambda \| u\|_{L^2} ) \leq C \| g\|_{L^2}.$$
Now
$$ (\Delta + \lambda^2) u =f \Rightarrow ( \Delta  + \lambda^2 + i \lambda 1_{\omega})  u =f + i 1_{\omega}  \lambda u,$$
which implies Proposition~\ref{helmoltz-weak}
 \section{Optimality of the logarithmic decay}\label{app.B}
In this section we give a simple example where we can ensure that the logarithmic decay in Theorem~\ref{th.2} cannot be improved.
Here we take $\Omega$ an arbitrary (non circular) ellipse.
$$ \Omega = \{ (x,y); ax^2+ by^2 \leq 1, \qquad 0< a<b\}$$ In this case it is well known (see the works by Keller and Rubinov~\cite{KeRu60} or more recently~\cite[Theorem 3.1]{NgGr13} that there exists a sequence of eigenfunctions $e_n$ associated to eigenvalues $\lambda^2_n\rightarrow + \infty $ which concentrate exponentially on the small axis of the ellipse, $I= \{0\} \times [-b^{-1}, b^{-1}]$. For all neighborhood $V$ of $I$,
\begin{equation}
\label{concentration} \| e_n \|_{L^2} =1, \qquad \int_{V^c} |e_n |^2 + |\nabla_x e_n | ^2 dx \leq C e^{-c \lambda_n}.
\end{equation} 
\begin{prop} Assume now that there exist a neighboorhood of $I$, $W$ such that $a\mid_{V}=0$. Then the exponential decay rate in Theorem~\ref{th.2} is optimal.
\end{prop}
Indeed, Let $\chi \in C^\infty_0 (W)$ equal to $1$ in a (smaller) neighborhood $V$ of $I$. The choice $i\lambda u_n = v_n = e_n$ gives 
\begin{equation}
(\mathcal{A} - i \lambda_n )\begin{pmatrix}u_n\\ v_n\end{pmatrix} = \begin{pmatrix} 0\\ g_n\end{pmatrix}, \| (u_n, v_n)\|_{H^1\times L^2} \sim 1,
\end{equation}
with 
$$ g_n =-i \lambda_n ^{-1} [\Delta , \chi ] e_n + \text{div} a \nabla (\chi e_n)= -i \lambda_n ^{-1} [\Delta , \chi ] e_n= O(e^{-c \lambda_n})_{L^2},
$$ where we used the fact that $[\Delta, \chi]$ is supported in $V^c$ and~\eqref{concentration} to get an exponential decay for $[\Delta , \chi ] e_n $, and the fact that $a$ vanishes on the support of $\chi$ which implies $ \text{div} a \nabla (\chi e_n)=0$.
This quasi-mode construction shows that the exponential factor in Proposition~\ref{resolv.2} cannot be improved, which in turn shows that the logarithmic factor in Theorem~\ref{th.2}  cannot be improved.

\def\cprime{$'$} \def\cprime{$'$}

\end{document}